\documentclass[]{amsart}

\usepackage{amsthm}
\usepackage[applemac]{inputenc}
\usepackage[english]{babel}
\usepackage{etex} %workspace
\usepackage{amsmath}
\usepackage{amsfonts}
\usepackage{amssymb}
\usepackage{wrapfig}

\usepackage{xcolor}  %for colored texts and individual colors 
\usepackage{extarrows}  %for larger and special arrows
\usepackage[pdfborder={0 0 0}]{hyperref}  %for Hyperlinks, pdfborder={0,0,0} removes red borderd hyperlinks 
\usepackage{enumitem} %enumeration with roman letters
\usepackage{graphicx} %to include graphics via "\includegraphics"
\allowdisplaybreaks %allows page breacks by using \displaybreak
\usepackage{array} %for thiker hlines in tables (cf. newcommand below)
\usepackage[font=footnotesize,labelfont=bf]{caption} %for captions of graphics not included as "figure"
\usepackage{longtable}
\usepackage{pdfcomment} %for comments in PDF Output
\usepackage{caption}
\usepackage{subcaption}
\usepackage{changes} %for crossed out text
\usepackage{pdfpages} %including pdf documents as pages
\usepackage{datetime}
\usepackage{todonotes}
\usepackage[backend=bibtex,style=numeric]{biblatex}
\AtEveryBibitem{%
	\clearfield{issn} % Remove issn
	\clearfield{doi} % Remove doi
	
	\ifentrytype{online}{}{% Remove url except for @online
		\clearfield{url}
	}
}

\newtheorem{thm}{Theorem}[section]
\newtheorem{prop}[thm]{Proposition}
\newtheorem{lem}[thm]{Lemma}
\newtheorem{cor}[thm]{Corollary}

\theoremstyle{definition}

\theoremstyle{remark}
\newtheorem{remark}[thm]{Remark}

\newcommand{\N}{\mathbb N}
\newcommand{\C}{\mathbb C}

\newcommand{\E}{\mathbb{E}}
\newcommand{\R}{\mathbb R}
\newcommand{\pr}{{\mathbb P}}
\newcommand{\var}{\operatorname{var}}

\newcommand{\gausspol}{G(N,\ell;z)}
\newcommand{\gausspolcoeff}{p(N,\ell,j)}

\newcommand{\verts}[1]{\vert{#1}\vert}

\newcommand{\FindStat}[1]{
	\ifx&#1&
	\url{www.findstat.org}    % #1 is empty
	\else \url{www.findstat.org/#1}    % #1 is nonempty
	\fi}

\begin{document}

\title[Limit theorems for polynomials]{Probabilistic limit theorems induced by\\ the zeros of polynomials}

%    author one information
\author{Nils Heerten}
\address{Faculty of Mathematics, Ruhr University Bochum}
%\curraddr{GG}
\email{nils.heerten@ruhr-uni-bochum.de}
\thanks{}

%    author two information
\author{Holger Sambale}
\address{Faculty of Mathematics, Ruhr University Bochum}
%\curraddr{FF}
\email{holger.sambale@ruhr-uni-bochum.de}
\thanks{}

%    author three information
\author{Christoph Th\"ale}
\address{Department of Mathematics, Ruhr University Bochum}
%\curraddr{EE}
\email{christoph.thaele@ruhr-uni-bochum.de}
\thanks{}

\subjclass[2020]{05A16, 52B20, 60C05, 60F05, 60F15}

\keywords{Berry-Esseen bound, combinatorial statistic, cumulant bound, mod-Gaussian convergence, zeros of polynomials}

\date{}

\dedicatory{}

\begin{abstract}
	Sequences of discrete random variables are studied whose probability generating functions are zero-free in a sector of the complex plane around the positive real axis. Sharp bounds on the cumulants of all orders are stated, leading to Berry-Esseen bounds, moderate deviation results, concentration inequalities and mod-Gaussian convergence. In addition, an alternate proof of the cumulant bound with improved constants for a class of polynomials all of whose roots lie on the unit circle is provided. A variety of examples is discussed in detail.
\end{abstract}

\maketitle

%\tableofcontents

% ================================================
\section{Introduction}
\label{ch:introduction}
% ================================================

Polynomials with non-negative coefficients are closely related to bounded $\N_0$-valued random variables, where $\N_0:=\{0,1,2,\ldots\}$, even by a one-to-one correspondence if the sum of the coefficients is assumed to be equal to $1$. Indeed, the probability generating function of any bounded $\N_0$-valued random variable is a polynomial with non-negative coefficients which sum up to $1$, and any such polynomial may be interpreted as the probability generating function of a random variable taking finitely many values all of which are non-negative integers.

It is therefore natural to search for connections between the location of the zeros of the polynomials and  asymptotic distributional properties of the corresponding random variables. For instance, if the polynomials only have real roots and thus can be written as the product of linear factors, the respective random variables can be represented as sums of independent Bernoulli variables in distribution. As a consequence, a central limit theorem is satisfied if and only if the variance tends to infinity. For a survey of this by now classical approach for proving probabilistic limit theorems, which was originally introduced by L.H.\ Harper \cite{Harper}, we refer to the article of J.\ Pitman \cite{pitman1997}.

For root-unitary polynomials, that is, polynomials all of whose roots lie on the unit circle in the complex plane, H.-K.\ Hwang and V.\ Zacharovas \cite{hwang_zacharovas2015} proved (among other results) that in typical cases, the limiting behaviour of sequences of the corresponding random variables $X_N$ is already determined by the limit of the fourth moment of the standardizations $X_N^\ast:=(X_N-\mathbb{E}X_N)/\sqrt{\var(X_N)}$. For example, they proved the remarkable fact that $X_N^\ast$ converges to the standard Gaussian random variable if and only if $\mathbb{E} (X_N^\ast)^4 \to~3$. This phenomenon can be regarded as an instant of what is known as a \textit{fourth-moment phenomenon}.

M.\ Michelen and J.\ Sahasrabudhe \cite{michelen2019Preprint, michelen2019} largely generalized this situation. They proved central limit theorems and Berry-Esseen bounds on the speed of convergence for sequences of bounded $\N_0$-valued random variables $X_N$ in terms of the variance and a quantity $\delta_N$ which depends on the ``geometry'' of the zero set of their generating functions $f_N$, which by construction are polynomials. For instance, in \cite[Theorem 1.1]{michelen2019Preprint} $\delta_N$ equals $\min_\zeta|1-\zeta|$ with the minimum running over all roots of the corresponding probability generating function $f_N$, while in Theorem 1.4 the polynomials are assumed to be zero-free in a small sector $\{z \in \C \colon |\text{arg}(z)| < \delta_N\}$, where $\text{arg}(z) \in (-\pi,\pi]$ stands for the argument of $z \in \C$.

As a by-product and technical device for their main results, in the proof of \cite[Lemma~8.1]{michelen2019Preprint} a bound on the cumulants of bounded $\N_0$-valued random variables $X_N$ is derived which in particular applies to the situation of \cite[Theorem~1.4]{michelen2019Preprint}. Writing $\sigma_N^2 := \var(X_N)$ and denoting the cumulant of order $m$ of $X_N^\ast:=(X_N-\E X_N)/\sigma_N$ by $\kappa_m(X_N^\ast)$, it says that
\begin{equation}
	\verts{\kappa_m(X_N^\ast)}\le {m!\over (c\delta_N\sigma_N)^{m-2}}
	\label{eq:introd_cum_bound_m&s}
\end{equation}
for some absolute constant $c>0$.

Cumulant bounds of this form contain a wealth of information on the asymptotic distributional properties of the random variables $X_N$, which goes far beyond of what has been studied in \cite{michelen2019Preprint}. This approach was put forward by L. Saulis and V.\,A.\ Statulevi\v{c}ius \cite{saulis1991} (see also the recent survey \cite{doring2021}). In the literature, a cumulant bound of the form
\begin{equation*}
	\verts{\kappa_m(X_N^\ast)}\le \frac{m!}{\Delta_N^{m-2}}\qquad (m \ge 3),
\end{equation*}
is known as the \emph{Statulevi\v{c}ius condition} and gives rise to a large number of strong limit theorems involving the quantity $\Delta_N$. In addition to central limit theorems, this includes Berry-Esseen bounds, moderate deviation principles and concentration inequalities. Moreover, the Statulevi\v{c}ius condition can also be used to establish mod-Gaussian convergence, a fairly new concept with far-reaching probabilistic consequences, which was introduced only within the last decade. For details about the role of cumulant bounds in mod-Gaussian convergence we refer to \cite[Chapter 5]{feray2016}. In particular, mod-Gaussian convergence leads to sharper versions of some of the bounds which can be derived by the methods established in \cite{saulis1991}. We remark that for the coefficients of real-rooted polynomials, mod-Gaussian convergence is a consequence of \cite[Theorem 8.1]{feray2016}.

In this note, we study cumulant bounds for sequences of bounded $\N_0$-valued random variables and the distributional limit theorems which follow from them together with a number of applications, mostly with a combinatorial flavour. In particular, since \eqref{eq:introd_cum_bound_m&s} is slightly hidden in the proofs of \cite{michelen2019Preprint} and requires substantial theoretical and technical background, we summarize the central arguments from \cite{michelen2019Preprint} in a short and largely self-contained proof. Another aim of this note is to provide an alternative and much more elementary proof of \eqref{eq:introd_cum_bound_m&s} for a certain class of root-unitary polynomials already appearing in \cite{hwang_zacharovas2015}, which at the same time leads to a far better absolute constant $c$ in this special situation. It is based on explicit representations of the cumulants $\kappa_m(X_N^\ast)$ of any order $m\geq 1$ for the random variables $X_N^\ast$ in \cite{hwang_zacharovas2015}.

\medspace

The organization of this note is as follows. In Section~\ref{ch:main_results} we state the cumulant bound from \cite{michelen2019Preprint} and various probabilistic limit theorems which follow as corollaries. Their theoretical background is summarized in the appendix in order to keep this paper largely self-contained. In Section~\ref{ch:special case} we consider a class of root-unitary polynomials appearing in \cite{hwang_zacharovas2015} and restate the cumulant bound in this situation with improved constants. A number of examples are discussed in Section~\ref{ch:examples}. In the last section we provide the proofs of the results from Section~\ref{ch:main_results} (Section~\ref{ch:proof_general_cumulant_bound}) and Section~\ref{ch:special case} (Section~\ref{ch:proof_of_sc_theorem}).

\section{Main results}
\label{ch:main_results}

Throughout this note, we shall consider sequences of bounded $\N_0$-valued random variables $(X_N)_N$, i.e., $X_N \in \{0,1,\ldots,n(N)\}$ for a suitable natural number $n(N)$. We denote their probability generating functions (which are polynomials by assumption) by $f_N:=f_{X_N}$ and their variances by $\sigma^2_N>0$. We further assume the generating functions $f_N$ to be zero-free in the sector
\begin{equation*}
	S(\delta_N):=\big\{z\in\C\setminus\{0\}:\arg(z)\in(-\delta_N,\delta_N)\big\}\qquad(\delta_N\in(0,\pi)),
\end{equation*}
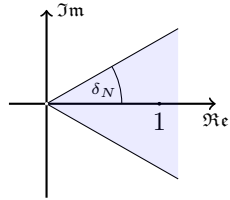
\begin{wrapfigure}{r}{0.5\textwidth}
	\centering
	\begin{tikzpicture}[scale=0.5]
		%GRID
		%		\draw[step=0.5cm,gray,very thin] (-1,-3) grid (5,3); %origin
		\filldraw[black] (3,0) circle (1pt) node[below]{$1$};
		\draw[->, thick] (-1,0)--(4.5,0) node[below, xshift=0]{\scriptsize $\mathfrak{Re}$}; %x-axis
		\draw[->, thick] (0,-2.5)--(0,2.5) node[right, xshift=0]{\scriptsize $\mathfrak{Im}$}; %y-axis
		\draw[-] (0,0)--(3.5,2);
		\draw[-] (0,0)--(3.5,-2);
		\draw (2,0) arc (0:29.74:2cm) node[midway,xshift=-0.2cm, yshift=-0.075cm]{{\tiny $\delta_N$}};
		\path [fill=blue, fill opacity = 0.08] (0,0) -- (3.5,2) -- (3.5,-2) -- (0,0);
		\filldraw[white] (0,0) circle (0.8pt) node[below,xshift=-0.2cm]{$0$};
	\end{tikzpicture}
	\caption[width=0.5\columnwidth]{The sector $S(\delta_N)$ (blue) in the complex plane.}
	\label{fig:sector_S_delta}
\end{wrapfigure}
adopting the notation used in \cite{michelen2019Preprint}. See Figure~\ref{fig:sector_S_delta} for a visualization.

We denote the cumulant of order $m$ of the random variable $X_N$ by $\kappa_{m,N}:=\kappa_m(X_N)$. The normalization of the random variables will be short-handed by $X_N^\ast:=(X_N-\E X_N)/\sigma_N$ and the respective cumulants of order $m$ of $X_N^\ast$ will be denoted by $\kappa_{m,N}^\ast:=\kappa_{m}(X_N^\ast)$. Throughout this note, $c$ will always indicate an absolute constant whose value might differ from occasion to occasion.

The first theorem provides a bound on the cumulants of random variables whose generating functions have no zeros in $S(\delta_N)$. As mentioned in the introduction, this result is taken from the proofs given in \cite{michelen2019Preprint} (see especially Lemma 8.1 therein). However, due to the variety of implications it draws, we highlight it at this point and recall the arguments of \cite{michelen2019Preprint} in a short and mostly self-contained proof in Section~\ref{ch:proof_general_cumulant_bound}.

\begin{thm}\label{thm:cumulant_bound}
	Let $(X_N)_{N\ge1}$ be a sequence of bounded random variables taking values in $\N_0$. For $\delta_N\in(0,\pi]$, assume the probability generating functions $f_N$ to have no roots lying in the sector $S(\delta_N)$. Then, for $m\ge3$,
	\begin{equation}\label{cumulantbound}
		\verts{\kappa_{m,N}^\ast}\le\frac{m!}{\Delta_N^{m-2}}
	\end{equation}
	with $\Delta_N := c\delta_N\sigma_N$, where $c>0$ is some absolute constant.
\end{thm}

As we will confirm in the proof, the arguments in \cite{michelen2019Preprint} lead to $c = 2^{-3248}$ for the constant in Theorem \ref{thm:cumulant_bound}.

The Statulevi\v{c}ius condition \eqref{cumulantbound} opens the door to a variety of remarkable probabilistic implications, of which we will discuss a few exemplarily. We briefly summarize these in Lemma~\ref{lem:appendix_1} in the appendix and refer to the monograph \cite{saulis1991} as well as the survey article \cite{doring2021} for further details. The first of them provides a Berry-Esseen type bound on the Kolmogorov distance of $X_N^\ast$ and a standard Gaussian random variable.

\begin{cor}
	\label{cor:clt_&_berry_esseen_bound}
	Let $(X_N)_{N\ge1}$ be a sequence of random variables as in Theorem \ref{thm:cumulant_bound}. Then, for a standard Gaussian random variable $Z$ and $\Delta_N$ as in \eqref{cumulantbound},
	\begin{equation*}
		\sup_{x\in\R}\big\vert \pr(X_N^\ast\le x)-\pr(Z\le x) \big\vert \le \frac{c_0}{\Delta_N}\qquad(N\ge1),
	\end{equation*}
	where $c_0>0$ is some absolute constant. 
	In particular, if $\delta_N\sigma_N\to\infty$ as $N\to\infty$, $X_N^\ast$ converges to $Z$	in distribution.
\end{cor}

By Lemma \ref{lem:appendix_1} \textit{(i)} (for $\gamma=0$), the constant in Corollary~\ref{cor:clt_&_berry_esseen_bound} can be chosen as $c_0=324\sqrt{2}$.

The cumulant bound \eqref{cumulantbound} also leads to moderate deviation results. In the sequel, we write $a_N=o(b_N)$ for two sequences of real numbers $(a_N)_{N\ge1}$ and $(b_N)_{N\ge 1}$ such that $a_N/b_N\to 0$ as $N\to\infty$.

\begin{cor}
	\label{cor:moderate_deviation}
	Let $(X_N)_{N\ge1}$ be a sequence of random variables as in Theorem \ref{thm:cumulant_bound}. Further, let $(a_N)_{N\ge1}$ be a sequence of real numbers with $a_N\to\infty$ as $N\to\infty$ and $a_N=o(\Delta_N)$ with $\Delta_N$ as in \eqref{cumulantbound}. Then,
	\begin{equation*}
		\lim_{N\to\infty}\frac{1}{a_N^2}\log\pr\big(X_N^\ast/a_N\in(x,\infty)\big)=-\frac{x^2}{2}\qquad(x\ge0).
		\label{eq:cor_moderate_deviation}
	\end{equation*}
\end{cor}

Corollary~\ref{cor:moderate_deviation} can be understood as a special case of a moderate derivation principle. In fact, under the same assumptions, a full moderate deviation principle (MDP) with speed $a_N^2$ and rate function $I(x)=x^2/2$ for the sequence $(X_N^\ast/a_N)_{N\ge1}$ is valid, see \cite[Theorem~3.1]{doring2021}. In our case, the sequence $(X_N^\ast/a_N)_{N\ge1}$ satisfies an MDP with
\begin{align*}
	\liminf_{N\to\infty}\frac{1}{a_N^2}\log\pr\big(X_N^\ast/a_N\in A\big)&\ge-\inf_{x\in A^\circ}\frac{x^2}{2},\\
	\limsup_{N\to\infty}\frac{1}{a_N^2}\log\pr\big(X_N^\ast/a_N\in A\big)&\le-\inf_{x\in \bar A}\frac{x^2}{2},
\end{align*}
for any Borel set $A\subset\R$, where $A^\circ$ and $\bar A$ denote the interior and the closure of $A$, respectively. This includes Corollary \ref{cor:moderate_deviation} as a special case.

Theorem \ref{thm:cumulant_bound} furthermore implies mod-Gaussian convergence for the sequence of random variables $X_N$, which is a special case of mod-$\phi$ convergence. We recall from \cite{feray2016,MeliotNik} that a sequence of random variables $X_N$ converges in the mod-Gaussian sense with parameters $t_N\to\infty$ and limiting function $\Psi$, provided that 
$$
\lim_{N\to\infty}\mathbb{E}[e^{zX_N}]e^{-t_Nz^2/2} = \Psi(z)
$$
locally uniformly on $\C$, where $\Psi$ is a non-degenerate analytic function.

\begin{cor}\label{cor:modGauss}
	Let $(X_N)_{N\ge1}$ be a sequence of random variables as in Theorem~\ref{thm:cumulant_bound} and $\Delta_N$ as in \eqref{cumulantbound}.
	\begin{itemize}
		\item[(i)] Assuming that $\kappa_{3,N}^\ast\Delta_N \to L \in \mathbb{R}$ as $N \to \infty$, the sequence of random variables $\Delta_N^{1/3}X_N^\ast$ converges in the mod-Gaussian sense with parameters $\Delta_{N}^{2/3}$ and limiting function $e^{Lz^3/6}$.
		\item[(ii)] Assuming that $\kappa_{3,N}^\ast = 0$ and $\kappa_{4,N}^\ast\Delta_N^2 \to L \in\R$ as $N \to \infty$, the sequence of random variables $\sqrt{\Delta_N}X_N^\ast$ converges in the mod-Gaussian sense with parameters $\Delta_{N}$ and limiting function $e^{Lz^4/24}$.
	\end{itemize}
\end{cor}

Natural situations in which part \textit{(ii)} applies include all polynomials discussed in Section~\ref{ch:special case}. We remark that mod-Gaussian convergence gives rise to various further implications as Cram\'er-Petrov type large deviations and precise deviations (see \cite[Chapter~5.2]{feray2016}). The latter means an estimate for probabilities of the type $\pr\big(X_N^\ast\in(x,\infty)\big)$ on a non-logarithmic scale for suitable values of $x$, which are allowed to depend on the parameter $N$ and which involve the limiting function $\Psi$. Since such results are rather technical to formulate, we refrain from presenting them. Instead, we mention the following concentration inequality, which yields a Bernstein-type bound for the upper tails of $X_N^\ast$.

\begin{cor}\label{cor:concentration_inequality}
	Let $(X_N)_{N\ge1}$ be a sequence of random variables as in Theorem~\ref{thm:cumulant_bound} and $\Delta_N$ as in \eqref{cumulantbound}. Then, there exists an absolute constant $C>0$ such that for all $x\ge0$,
	\begin{equation*}
		\pr(X_N^\ast\ge x)\le C\exp\Big( -\frac{1}{2}\frac{x^2}{2+x/\Delta_N} \Big).
	\end{equation*}
\end{cor}

% ================================================
\section{Results for a class of root-unitary polynomials}
\label{ch:special case}
% ================================================

All over this section, we consider $\N_0$-valued random variables $X_N$ whose probability generating functions admit a representation of the form
\begin{equation}
	f_N(z)=f_{X_N}(z):=\frac{P_N(z)}{P_N(1)}\qquad(z\in\R),
	\label{eq:sc_prob_gen_fct}
\end{equation} 
where 
\begin{equation}
	P_N(z):=\prod_{1\le j\le N}\frac{1-z^{b_j}}{1-z^{a_j}},
	\label{eq:sc_gen_polynomial}
\end{equation}
for some natural number $N$ and exponents $a_j,b_j\in \N$, $j=1,\dots,N$. We assume that $b_j\ge a_j$ for all $j=1,\dots,N$ and note that $P_N(1)=\prod_{1\le j\le N}b_j/a_j$, which follows readily from the fact that $(1-z^k)=(1-z)(1+z+\ldots+z^{k-1})$ for $k\ge1$. The degree of $P_N$ is denoted by $n$ and is related to the parameter $N$ by $n=\sum_{1\le j\le N}(b_j-a_j)$. As opposed to the probability generating function \eqref{eq:sc_prob_gen_fct}, we refer to $P_N$ as the generating polynomial of $X_N$ for the remainder of this section.

The probability generating functions $f_N$ defined by \eqref{eq:sc_prob_gen_fct} have roots all located on the unit circle in the complex plane, i.e., they are \emph{root-unitary} in the terminology of \cite{hwang_zacharovas2015}. More precisely, their roots are a subset of $\{\exp(\pm 2\pi i k/b_j), k=1, \ldots, \lfloor b_j/2 \rfloor, j=1, \ldots, N\}$, $1$ is never a root (since $P_N(1) \ne 0$) and we may assume that $\max_j a_j < \max_j b_j$. In particular, the functions $f_N$ have no zeros in $S(\delta_N)$ for $\delta_N := 2\pi/\max_{j=1,\ldots,N} b_j$.

Probability generating functions of type \eqref{eq:sc_prob_gen_fct} have been studied in \cite[Section 4]{hwang_zacharovas2015}, where they form an important subclass of root-unitary polynomials, and previously also in \cite[Section 3]{chen2008} in the more specific context of $q$-Catalan numbers. An extensive discussion in particular related to central limit theorems and a large number of examples where polynomials of type \eqref{eq:sc_gen_polynomial} naturally appear can be found in \cite{billey2023cyclotomic}, where they are called cyclotomic generating functions. We emphasize that we use a different notation as in \cite{hwang_zacharovas2015}, who index \eqref{eq:sc_prob_gen_fct} and \eqref{eq:sc_gen_polynomial} by the degree $n$. Random variables with probability generating functions of type \eqref{eq:sc_prob_gen_fct} naturally appear in a variety of examples, a few of which will be discussed in Section~\ref{ch:examples_special case}.

A reformulation of Theorem~\ref{thm:cumulant_bound} for the probability generating functions \eqref{eq:sc_prob_gen_fct} admits a more straightforward proof by elementary means, which we present in Section~\ref{ch:proof_of_sc_theorem}. Its core ingredient is an explicit representation of the cumulants of the corresponding random variables $X_N$. For $m\ge3$, they are given by
\begin{equation}
	\kappa_{m,N}=\frac{B_m}{m}\sum_{1\le j\le N}(b_j^{m}-a_j^{m}),
	\label{eq:sc_cumulant_representation}
\end{equation}
as was derived in \cite{hwang_zacharovas2015}, also see \cite{billey2023cyclotomic}. Here, $B_{m}$ denotes the $m$-th Bernoulli number. In particular, since $B_{2m+1}=0$ for any $m \ge 1$, all cumulants of odd order $\ge 3$ vanish identically, and when establishing cumulant bounds it is therefore sufficient to consider the cumulants of even orders $2m$. 

\begin{thm}\label{thm:sc_main_theorem}
	Let $P_N(z)$ be a polynomial as in \eqref{eq:sc_gen_polynomial} and $X_N$ a random variable defined as in \eqref{eq:sc_prob_gen_fct}. Then, for $m\ge2$, the $2m$-th cumulant of $X_N^\ast$ satisfies
	\begin{equation*}
		\vert\kappa_{2m,N}^\ast\vert\le\frac{(2m)!}{\Delta_{N}^{2m-2}}
%		\label{eq:statul_cond}
	\end{equation*}
	with
	\begin{equation*}
		\Delta_N=\pi^2\sqrt{\frac{7}{6}}\frac{\sigma_N}{M_N},
	\end{equation*}
	where $M_N=\max_{1\le j\le N}b_j$. In particular, Corollaries~\ref{cor:clt_&_berry_esseen_bound}--\ref{cor:concentration_inequality} hold with this choice of $\Delta_N$.
\end{thm}

As noted above, we may set $\delta_N = 2\pi/M_N$, so that in the notation of Section~\ref{ch:main_results}, we have $\Delta_N= c\delta_N\sigma_N$ with $c = \pi \sqrt{7/24}$. In particular, this gives back \eqref{cumulantbound} but with a remarkably better absolute constant.

Note that \cite[Lemma 4]{hwang1998} also yields Berry-Esseen type bounds for sequences of random variables whose probability generating functions factorize into polynomials with non-negative coefficients. In particular, assuming all the factors $(1-z^{b_j})/(1-z^{a_j})$ in \eqref{eq:sc_gen_polynomial} to be polynomials with non-negative coefficients, by \cite[Lemma 4]{hwang1998} a Berry-Esseen bound
\[
	\sup_{x\in\R}\big\vert \pr(X_N^\ast\le x)-\pr(Z\le x) \big\vert \le \frac{c}{\Delta_N'}
\]
holds with $\Delta_N = \sigma_N/M_N'$, where $M_N' := \max_{1 \le j \le N} (b_j-a_j)$. Note that in all examples discussed in Section~\ref{ch:examples_special case}, the asymptotic behaviour of $M_N$ and $M_N'$ (whenever applicable) coincides. On the other hand, some of the examples we treat do not fall into the class of random variables considered in \cite{hwang1998}.

% ================================================
\section{Examples}
\label{ch:examples}
% ================================================

\subsection{Polynomials with roots only on the negative half-axis}
\label{ch:negroots}

One of the simplest situations in which our results apply are polynomials which have only real roots, all of which are negative (so that they are zero-free in $S(\delta)$ for any $\delta \in (0,\pi]$). As already mentioned in the introduction, central limit theorems for such situations are rather classical and go back to the work of L.H.\ Harper \cite{Harper} and many examples can be found, for example in \cite{pitman1997}. We shall exemplarily discuss one further application, which deals with descents in finite Coxeter groups, a topic we will  pick up below once again. 

Given a permutation $\pi$ on the set $\{1,\dots,N\}$, the descents of $\pi$ are defined as 
\begin{equation*}
	\text{Des}(\pi)=\big\{i:1\le i<N, \pi(i)>\pi(i+1)\big\}.  
\end{equation*}
Descents of permutations can be regarded as a special type of the more general theory of descents in finite Coxeter groups. For a detailed introduction to Coxeter groups we refer to \cite[Chapter 1]{bjorner2006}. For simplicity, we only consider the following three classical types.

The Coxeter group of type $A_N$ corresponds to the permutations on $\{1, \ldots, N+1\}$. The Coxeter group of type $B_N$ can be realized as the group of signed permutations, that is, the group of all bijections $\pi$ on $\{\pm1,\dots,\pm N\}$, such that $\pi(-i)=-\pi(i)$, see \cite[Chapter~8.1]{bjorner2006}. Following the one-line notation of \cite[Section~2]{kahle2020}, we write $\pi=[\pi(1),\dots,\pi(N)]$, with $\pi(i)\in\{\pm1,\dots,\pm N\}$ and $\{ \vert\pi(1)\vert,\dots,\vert\pi(N)\vert\}=\{1,\dots,N\}$. The Coxeter group of type $D_N$ can then be realized as a subgroup of $B_N$ with an even number of negative entries in the one-line notation, in other words,
\begin{equation*}
	D_N=\big\{ \pi\in B_N\::\:\pi(1)\pi(2)\cdots\pi(N)>0 \big\},
\end{equation*}
see \cite[Chapter~8.2]{bjorner2006}. Note that the groups $A_N, B_N$ and $D_N$ all have rank $N$.

In accordance with Proposition 8.1.2 and Proposition 8.2.2 in \cite{bjorner2006} we set
\begin{equation*}
	\pi(0):=\begin{cases}
		0 & \text{if }\pi\in A_{N-1}\text{ or } B_N, \\
		-\pi(2) & \text{if }\pi\in D_N.
	\end{cases}
\end{equation*}
The descents of an element $\pi\in A_{N-1},B_N,D_N$ can now be defined as
\begin{equation*}
	\text{Des}(\pi):=\big\{i:0\le i<N, \pi(i)>\pi(i+1)\big\},
\end{equation*}
see \cite[Section 2]{kahle2020}. We refer to \cite{meier2022} for a discussion of the more general concept of $d$-descents in the classical types of Coxeter groups.

If we set $W_N$, $W\in\{A,B,D\}$, to be a finite Coxeter group of rank $N$, the number of elements in $W_N$ with exactly $k\in\{0,1,\dots,N\}$ descents is called the $W_N$-Eulerian number, see \cite{kahle2020}. If we define a random variable $X_N$ as the number of descents of an element of $W_N$ chosen uniformly at random, the distribution of $X_N$ is called the $W_N$-Eulerian distribution. The generating function of $X_N$ is given~by
\begin{equation*}
	f_N(z)= \frac{P_N(z)}{P_N(1)},\qquad \text{where } P_N(z) = \prod_{i=1}^N(z+q_i)
\end{equation*}
for some positive real numbers $q_1,\dots,q_N$ (see \cite[Theorem 2.4]{kahle2020}) and hence only has real-valued, negative roots.

\begin{prop}
	For the Coxeter group of type $A_N$, the assumptions of Theorem~\ref{thm:cumulant_bound} are satisfied with $\delta_N = \pi$ and $\sigma_N= (N+2)/12$, so that $\Delta_N = \mathcal{O}(\sqrt{N})$. In particular, $\sqrt{\Delta_N}X_N^\ast$ converges in the mod-Gaussian sense with parameters $\Delta_N$ and limiting function $\Psi(z) = \exp(-c^2\pi^2z^4/240)$, where $c$ is the absolute constant from Theorem~\ref{thm:cumulant_bound}.
\end{prop}

Indeed, according to \cite[Theorem~5.3]{janson2013}, the cumulants of order $m$ are given by 
\begin{equation*}
	\kappa_{m,N}=(N+2){B_m\over m}\qquad(2\le m\le N),
\end{equation*}
where $B_m$ denotes the $m$-th Bernoulli number. In particular, $\kappa_{4,N}=-(N+2)/120$ and $\kappa_{4,N}\Delta_N^2\to-\pi^2c^2/10$ as $N\to\infty$, so that Corollary~\ref{cor:modGauss} \textit{(ii)} can be applied.

Similar results also hold for $W_N=B_N$ and $W_N=D_N$. For instance, note that we always have $\sigma_N^2\sim N/12$ by \cite[Corollary~4.2]{kahle2020}, and hence $\Delta_N= \mathcal{O}(N^{-1/2})$. The corresponding central limit theorems Berry-Esseen bounds are consistent with \cite[Theorem 6.2]{kahle2020} and \cite[Corollary 2.7]{meier2022}.

In fact, as mentioned in \cite[Theorem 2.4]{kahle2020}, the representation of the generating function $f_N(z)$ given above holds for descents in finite Coxeter groups in general and hence all consequences remain valid in such a more general set-up. However, a representation of descents in general finite Coxeter groups would require the study of root systems of these groups and thus significantly complicate the presentation. We refer to \cite{kahle2020} and \cite{meier2022} for a detailed discussion.

\begin{figure}[t!]
	\centering
	\begin{subfigure}{0.5\textwidth}
		\centering
		\includegraphics[scale=0.4]{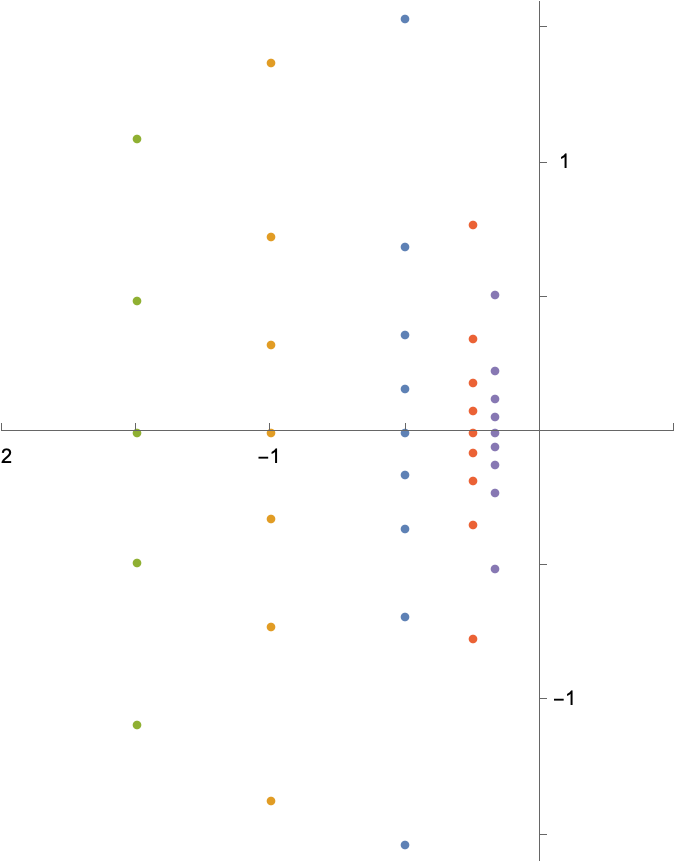}
	\end{subfigure}%
	\begin{subfigure}{0.5\textwidth}
		\centering
		\includegraphics[scale=0.4]{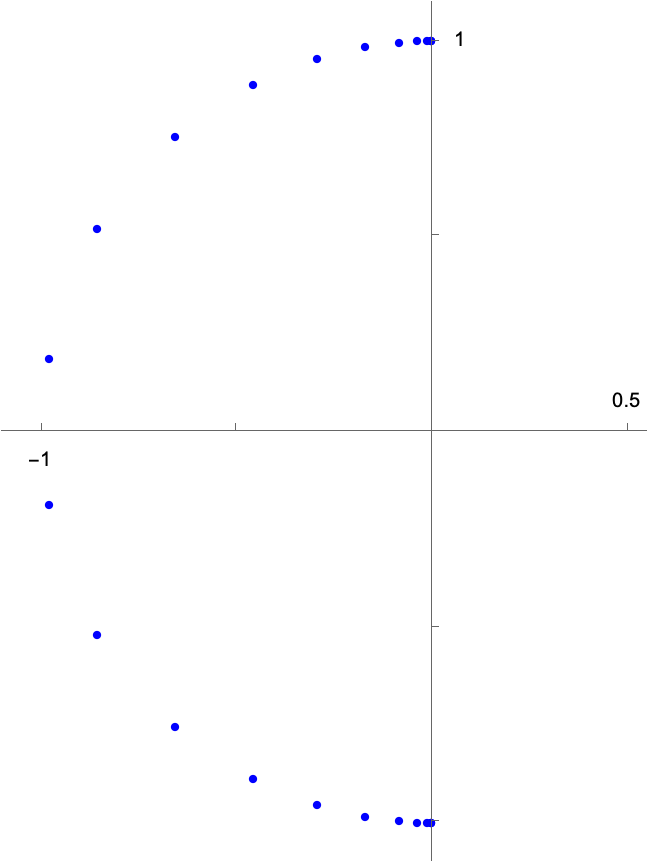}
	\end{subfigure}%
	\caption{Location of the roots of the generating polynomials of the conditional binomial distribution (left panel) with $N=25$, $p=1/2$ (blue), $p=1/3$ (yellow), $p=1/4$ (green), $p=2/3$ (orange) and $p=3/4$ (violet) and of the alternating descents (right panel) with $N=25$ in the complex plane.}
	\label{fig:root_plot_binom_&_alt_desc}
\end{figure}

\subsection{Hurwitz polynomials}
A Hurwitz (or Hurwitz-stable) polynomial is a polynomial all of whose roots $\zeta$ lie in the open left half-plane, i.e., $\mathfrak{Re}(\zeta)<0$. Such a polynomial is again zero-free in $S(\delta)$ for any $\delta \in (0,\pi/2)$. In particular, the polynomials discussed in Section~\ref{ch:negroots} form a subclass of Hurwitz polynomials. In the sequel, we discuss several examples of Hurwitz polynomials with non-real roots. 

\subsubsection{Conditional binomial distribution}
\label{ch:ex:gen_binom_distr}

We say that a random variable $X_N$, $N\in\N$, has conditional binomial distribution $\operatorname{CoBin}(N,p)$ with parameters $N$ and $p \in (0,1)$, if for all $k = 0, \ldots, N-1$,
\begin{equation*}
	\pr(X_N=k):=\pr(Y_N=k\mid Y_N\le N-1),
\end{equation*}
where $Y_N\sim\operatorname{Bin}(N,p)$ has binomial distribution with parameter $p$. This slightly generalizes the random variables discussed in Lemma 4 and Theorem 3 in \cite{belovas2020} (which correspond to the case of $p=1/2$). For each $N$, $X_N$ has probability generating function 
\[
	f_N(z)={1\over 1-p^N}\sum_{k=0}^{N-1} z^k\binom{N}{k} p^k(1-p)^{N-k}
	={\big(p(z-1)+1\big)^N-(zp)^N\over 1-p^N}.
\]
To see that $f_N(z)$ is a Hurwitz polynomial, note that the equation $f_N(z) = 0$ is equivalent to $0 = (zp+(1-p))^N-(zp)^N$, which leads to solving
\begin{equation*}
	\Big(1+\frac{1-p}{zp}\Big)^N=1.
\end{equation*}
Consequently, we obtain
\[
1+\frac{1-p}{z_kp} = \cos\Big(\frac{2\pi k}{N}\Big) + i \sin\Big(\frac{2\pi k}{N}\Big),\quad k=0,\ldots,N-1,
\]
which yields
\begin{equation*}
	z_k=\frac12 \Big(\frac1p-1\Big)\Big(-1-i\cot\Big(\frac{\pi k}{N}\Big)\Big),\qquad k=1,\dots,N-1, 
\end{equation*}
and therefore all roots $\zeta$ of $f_N(z)$ satisfy $\mathfrak{Re}(\zeta)=(1-1/p)/2<0$. The location of these roots for various $p$ is shown in the left panel of Figure~\ref{fig:root_plot_binom_&_alt_desc}.

\begin{prop}
	For any fixed $p \in (0,1)$, the assumptions of Theorem~\ref{thm:cumulant_bound} are satisfied with $\delta_N=\pi/2$ and $\sigma_N^2=Np(1-p)+o(1)$, so that $\Delta_N=\mathcal{O}(\sqrt{N})$. In particular, $\Delta_N^{1/3}X_N^\ast$ converges in the mod-Gaussian sense with parameters $\Delta_N^{2/3}$ and limiting function $\Psi(z)= \exp((1-2p)c\pi z^3/12)$, where $c$ is the absolute constant from Theorem~\ref{thm:cumulant_bound}.
\end{prop}

Indeed, it is straightforward to show that
\begin{align*}
	\E X_N=N\frac{p-p^N}{1-p^N}\qquad\text{and}\qquad \sigma_N^2=N\frac{p(1-p)}{1-p^N}-N^2\frac{p^N((p-1)^2)}{(1-p^N)^2}.
\end{align*}
Moreover, 
\begin{equation*}
	\kappa_{3,N}=Np(1-p)(1-2p)+o(1),
\end{equation*}
which yields $\kappa_{3,N}^\ast\sigma_N=\kappa_{3,N}\sigma_N^{-2}\rightarrow1-2p$ as $N\to\infty$, so that mod-Gaussian convergence follows from Corollary~\ref{cor:modGauss} \textit{(i)}.

\subsubsection{Ehrhart polynomials}

Taking $p=1/2$ in the previous example, we saw that all roots of the probability generating function $f_N$ of $X_N$ were located on the line $\{z\in\C:\mathfrak{Re}(z)=-1/2\}$. A similar property has become rather prominent in the theory of lattice polytopes of which we can outline only selected aspects and refer to \cite{beck_robins2015} for further details. We recall that a polytope $Q\subset\R^N$ is called a lattice polytope if all of its vertices have integer coordinates. A lattice polytope $Q$ is reflexive if also its convex dual $Q^\circ:=\{x\in\R^N:\langle x,y\rangle\leq 1\ \text{for all}\ y\in Q\}$ is a lattice polytope, where $\langle\,\cdot\,,\,\cdot\,\rangle$ denotes the standard scalar product in $\R^N$. It is well known and the content of Ehrhart's theorem \cite[Theorem 3.8]{beck_robins2015} that for a lattice polytope $Q$ the counting function
$$
E_Q(k):=|\{t\in\mathbb{Z}^N:t\in kQ\}|,\qquad k\in\N_0,
$$
is the evaluation of a polynomial $E_Q(z)$, $z\in\R$, of degree $N$, the so-called Ehrhart polynomial of $Q$. To interpret such a polynomial as the probability generating polynomial of a random variable, we will assume from now on that $Q$ is Ehrhart positive, meaning that all coefficients of $E_Q(z)$ are positive. Finally, a lattice polytope $Q$ is called a CL-polytope, provided that all roots of its Ehrhart polynomial are located on the so-called canonical line $\{z\in\C:\mathfrak{Re}(z)=-1/2\}$. 

For example, if $W_N:=[0,1]^N$ is the $N$-dimensional unit cube, then
$$
E_{W_N}(z)=(z+1)^N,
$$
see \cite[Theorem 2.1]{beck_robins2015}. From this we conclude that the random variable with probability generating function $E_{W_N}(z)/E_{W_N}(1)=(z+1)^N/2^N$ follows a binomial distribution with parameter $N$ and $1/2$. However, since $z=-1$ is the only root of $E_{W_N}$, $W_N$ is not a CL-polytope, but it corresponds to a real-rooted polynomial as discussed in Section~\ref{ch:negroots}.

Other examples connected to the binomial distribution are root polytopes of type $A$ and $C$ and their convex duals. If $e_1,\ldots,e_N$ denotes the standard orthonormal basis in $\R^N$ the root polytopes of type $A$ and $C$ are defined, respectively, as the convex hulls of the classical root systems of type $A$ and $C$, that is, 
\begin{align*}
	A_N &:= \mathrm{conv}(\{\pm (e_i+\ldots+e_j):1\leq i\leq j\leq N\}),\\
	C_N &:= \mathrm{conv}(\{\pm (2e_i+\ldots+2e_{N-1}+e_N):1\leq i\leq N-1\}\cup\{\pm e_N\}).
\end{align*}
The Ehrhart functions of both $A_N$ and $C_N$ are known explicitly, showing in particular that both polytopes are of class CL, see \cite[Sections 3 and 5]{higashitani_kummer_michalek2017} . For the dual convex polytopes $A_N^\circ$ and $C_N^\circ$ from \cite[Lemma 5.3]{higashitani_kummer_michalek2017} and \cite[Theorem 1.1]{higashitani_yamada2021} one has that
$$
E_{A_N^\circ}(z) = (z+1)^{N+1} - z^{N+1}\qquad\text{and}\qquad E_{C_N^\circ}(z) = (z+1)^{N} + z^{N},
$$
implying that also $A_N^\circ$ and $C_N^\circ$ are CL-polytopes. Moreover, up to the normalization constant $E_{A_{N-1}^\circ}(1)=2^N-1$, $E_{A_{N-1}^\circ}(z)$ coincide with the probability generating function $f_N$ for $p=1/2$ considered in the previous section. In other words, the random variable connected to $A_N^\circ$ has the conditional binomial distribution $\mathrm{CoBin}(N,1/2)$. For a probabilistic interpretation of $E_{C_N^\circ}(z)$, let $U_N$ and $V_N$ be independent random variables, $U_N$ having a binomial distribution with parameters $N$ and $1/2$, while $V_N$ has a Bernoulli distribution with success probability $2^N/(2^N+1)$. Consider the random variable $N-U_NV_N$. By conditioning on the value of $V_N$, one easily check that its probability generating function coincides with $E_{C_N^\circ}(z)/E_{C_N^\circ}(1)$.

The Ehrhart polynomials of CL-polytopes are thus Hurwitz polynomials and the random variables generated by the coefficients of the normalizations of these polynomials satisfy the asymptotic distributional results derived in Section \ref{ch:main_results}.

%Bitte folgende Referenzen raussuchen und einfügen:
%Higashitani/Yamada: THE DISTRIBUTION OF ROOTS OF EHRHART POLYNOMIALS FOR THE DUAL OF ROOT POLYTOPES OF TYPE C
%Higashitani/Kummer/Michalek: INTERLACING EHRHART POLYNOMIALS OF REFLEXIVE POLYTOPES
%BeckBuch: Beck and Robins, Computing the Continuous Discretely

\subsubsection{Alternating descents in permutations}
\label{ch:ex:Alternating descents}

Recalling the definition of descents of a permutation $\pi$ on the set $\{1,\dots,N\}$ from Section~\ref{ch:negroots}, the alternating descents of $\pi$ are the positions $i$ such that $\pi(i)>\pi(i+1)$ if $i$ is odd, or $\pi(i)<\pi(i+1)$ if $i$ is even. %\todonew{OEIS zitieren!}
That is, an alternating descent is a descent if $i$ is odd and an ascent if $i$ is even. It was shown in \cite{chebikin2008} that the number of alternating descents agrees with the number of $3$-descents of permutations on $\{1,\ldots, N+1\}$ with $\pi(1) = 1$. Here, a $3$-descent is an index $i$ such that $\pi(i)\pi(i+1)\pi(i+2)$ forms an odd permutation of size $3$, i.e., it has one of the patterns $132$, $213$, or $321$.

Let $X_N$ be the number of alternating descents of a permutation $\pi$ on $\{1,\ldots,N\}$ chosen uniformly at random. In \cite[Theorem 4]{ma2015}, it was shown that all the roots $\zeta$ of the generating function $f_N(z)$ of $X_N$ satisfy $\verts{\zeta}=1$ and have a strictly negative real part. That is, they all lie on the left half on the unit circle in the complex plane, see the right panel of Figure~\ref{fig:root_plot_binom_&_alt_desc}. In other words, $f_N(z)$ is a Hurwitz polynomial.

\begin{prop}
	The assumptions of Theorem~\ref{thm:cumulant_bound} are satisfied with $\delta_N=\pi/2$ and $\sigma_N^2 =(5N+3)/12$, so that $\Delta_N= \mathcal{O}(\sqrt{N})$. In particular, $\sqrt{\Delta_N}X_N^\ast$ converges in the mod-Gaussian sense with parameters $\Delta_N$ and limiting function $\Psi(z) = \exp(-27\pi^2 c^2z^4/1600)$, where $c$ is the absolute constant from Theorem~\ref{thm:cumulant_bound}.
\end{prop}

Indeed, note that $\kappa_{3,N}=0$ and $\kappa_{4,N}=-(81N-78)/120$ for $N\ge3$ as can be computed from the explicit exponential generating function reported in \cite[Equation (1)]{ma2015} or \cite[Section 5.4.2]{hwang_chern_duh2020}. This leads to $\kappa_{4,N}^\ast\Delta_N^2\to-\frac{81c^2\pi^2}{200}$ as $N\to\infty$, which establishes mod-Gaussian convergence by Corollary~\ref{cor:modGauss} \textit{(ii)}.

In particular, our results confirm the speed of convergence in the central limit theorem established in \cite[Section 5.4.2]{hwang_chern_duh2020}.

% ================================================
\subsection{Root unitary polynomials}
\label{ch:examples_special case}
% ================================================

In this section, we discuss various examples of polynomials which are zero-free in some sector $S(\delta_N)$ with $\delta_N \to 0$ as $N \to \infty$. In particular, these examples provide natural appearances of generating functions that are root-unitary polynomials as studied in Section~\ref{ch:special case}.

\subsubsection{Inversions in Coxeter groups}
\label{ex:general_coxeter_groups}

\begin{figure}[t!]
	\centering
	\begin{subfigure}{0.5\textwidth}
		\centering
		\includegraphics[scale=0.4]{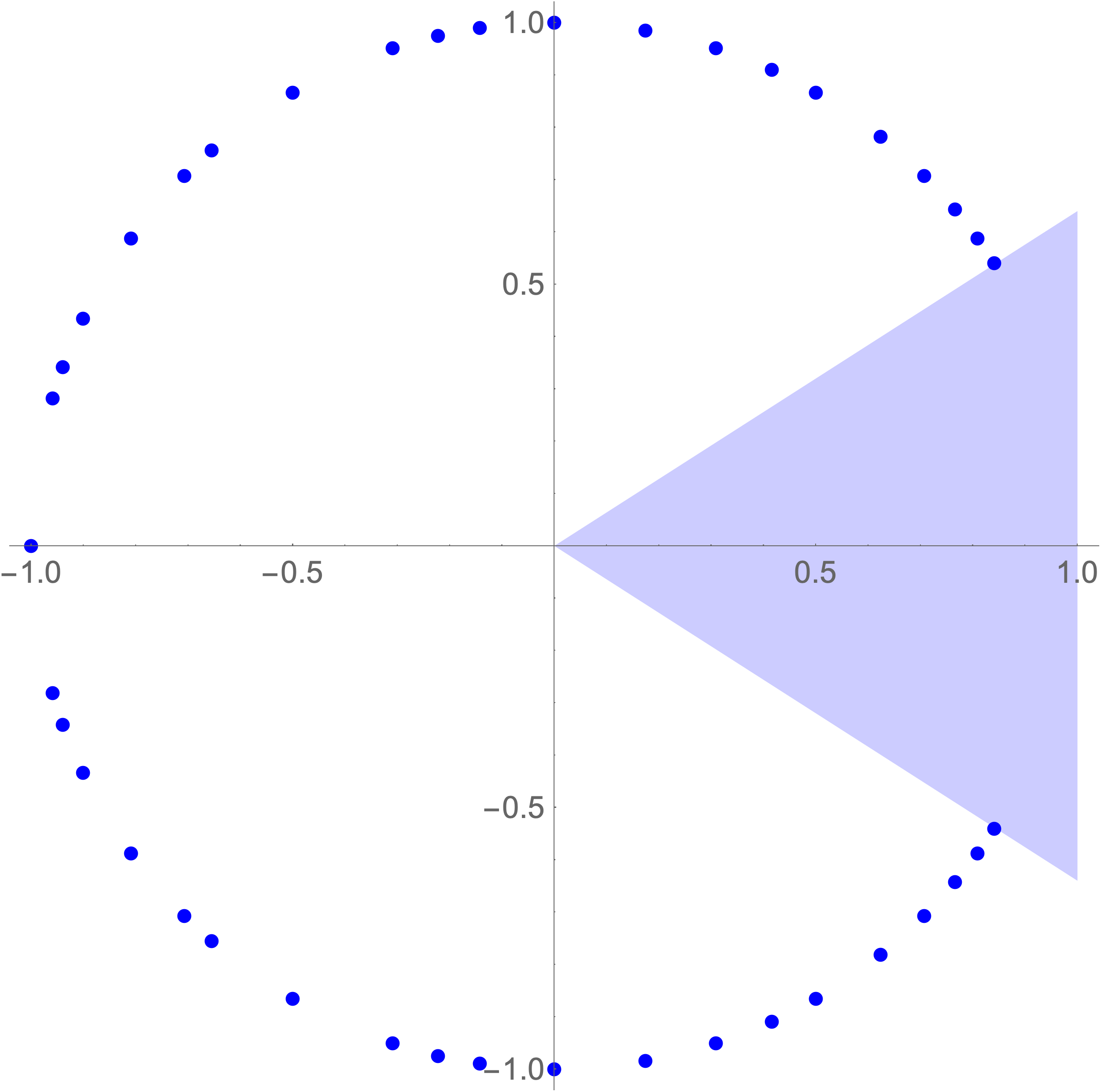}
	\end{subfigure}%
	\begin{subfigure}{0.5\textwidth}
		\centering
		\includegraphics[scale=0.4]{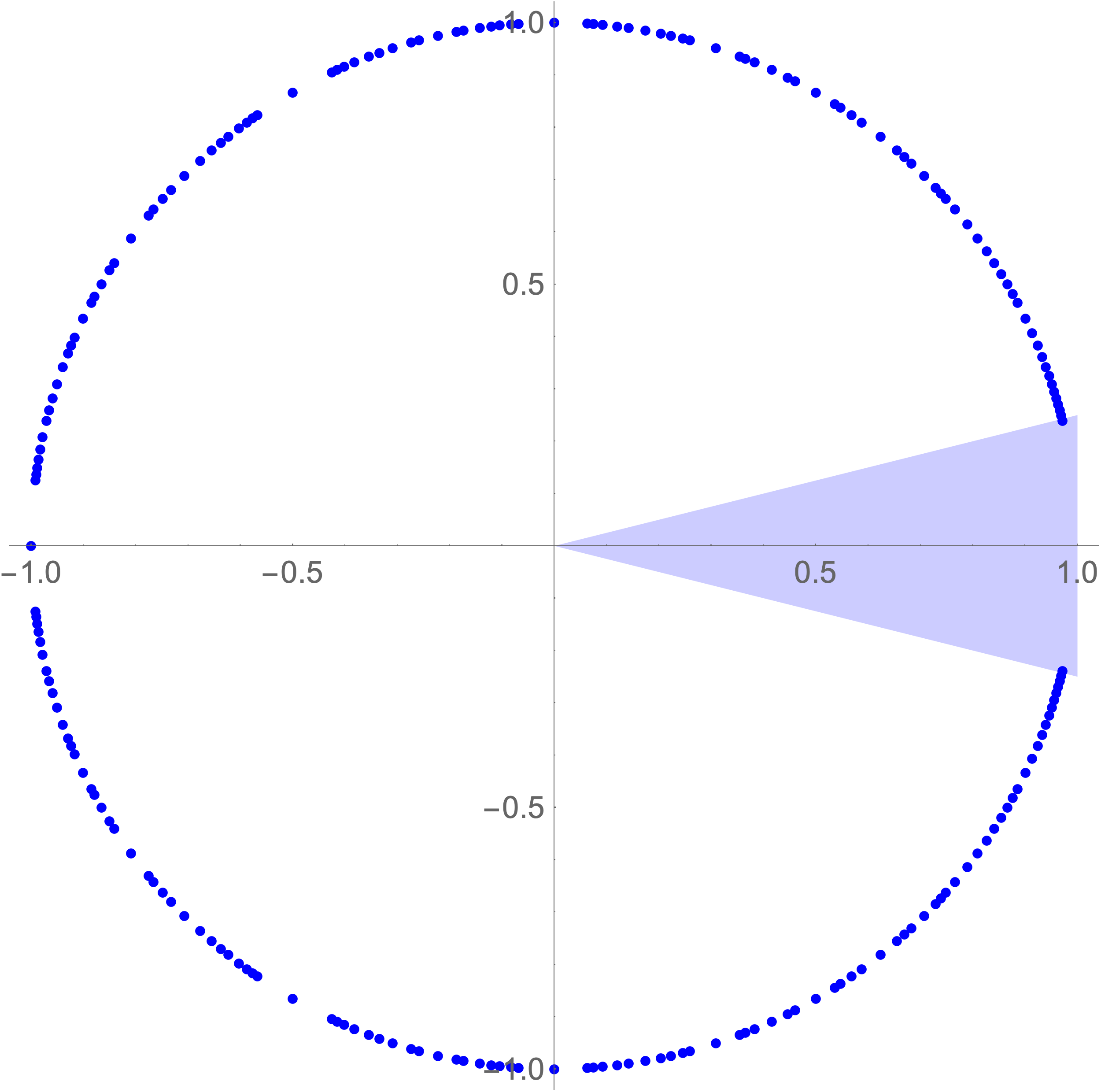}
	\end{subfigure}%
	\caption{Location of the roots of the generating polynomial of the number of inversions in a Coxeter group of type $A_N$ with $N=10$ (left panel) and $N=25$ (right panel) in the complex plane. The sector $S(\delta_N)$ is shaded in blue.}
	\label{fig:root_plot_inversions}
\end{figure}

An inversion of a permutation $\pi$ on the set $\{1,\dots,N\}$ is a pair $(i,j)$ with $1\le i<j\le N$ and $\pi(i)>\pi(j)$. Similar to descents in Section~\ref{ch:negroots}, inversions of permutations can be regarded as a special type of the more general theory of inversions in finite Coxeter groups, where as above we only consider the three classical types $A_N,B_N$ and $D_N$. 
For any of these three groups, certain elements may be identified as inversions. Indeed, according to the notation in \cite[Section~2]{kahle2020}, we set
\begin{align*}
	\operatorname{inv}^+(\pi)&:= \{ 1\le i<j\le N:\pi(i)>\pi(j) \},\\
	\operatorname{inv}^-(\pi)&:= \{ 1\le i<j\le N:-\pi(i)>\pi(j) \},\\
	\operatorname{inv}^\circ(\pi)&:= \{ 1\le i\le N:\pi(i)<0 \}.
	\intertext{Then, the inversions in an element $\pi \in W_N$, $W \in \{A,B,D\}$ are defined as}
	\operatorname{inv}_{A_{N-1}}(\pi)&:=\operatorname{inv}^+(\pi),\\
	\operatorname{inv}_{B_N}(\pi)&:=\operatorname{inv}^+(\pi)\cup\operatorname{inv}^-(\pi)\cup\operatorname{inv}^\circ(\pi),\\
	\operatorname{inv}_{D_N}(\pi)&:=\operatorname{inv}^+(\pi)\cup\operatorname{inv}^-(\pi).
\end{align*}%\operatorname{inv_{A_{N-1}}}
We also refer to \cite[Section~2]{meier2022} for a discussion of more general $d$-inversions in the classical types of Coxeter groups.

The numbers of elements in $W_N$, $W\in\{A,B,D\}$, with exactly $k\in\{0,1,\ldots,N\}$ 
inversions are called the $W_N$-Mahonian numbers. If we define a random variable $X_N:=X(W_N)$ as the number of inversions of an element of $W_N$ chosen uniformly at random, the distribution of $X_N$ is called the $W_N$-Mahonian distribution. The generating polynomial for $X_N$ is given by
\begin{equation}
	P_N(z)=\prod_{1\le j\le N}\frac{1-z^{d_j}}{1-z}\qquad(z\in\R),
	\label{eq:example_coxeter_generating_function}
\end{equation}
where $d_1, \ldots, d_N$ are the degrees of $W_N$, see \cite[Theorem~7.1.5]{bjorner2006} and note that the exponents $e_j$ of  Coxeter groups therein are related to the degrees by $d_j=e_j+1$. The degrees of the three types of Coxeter groups we consider are summarized in the following table.

\begin{table}[h]
	\renewcommand{\arraystretch}{1.5}
	\centering
	\begin{tabular}{c|l}
		Type&Degrees\\
		\hline
		$A_N$&$d_1=2,d_2=3,\dots,d_N=N+1$\\
		$B_N$&$d_1=2,d_2=4,\dots,d_N=2N$\\
		$D_N$&$d_1=2,d_2=4,\dots,d_{N-1}=2N-2,d_N=N$
	\end{tabular}
\end{table}
\renewcommand{\arraystretch}{1.0}

In particular, the polynomials \eqref{eq:example_coxeter_generating_function} are of the form \eqref{eq:sc_gen_polynomial} with $b_j=d_j$, so that $M_N=\max_{1\le j\le N}d_j$, and $a_j=1$ for $j\in\{1,\ldots,N\}$, and have degree $n=\sum_{1\le j\le N}(d_j-1)$. We refer to Figure~\ref{fig:root_plot_inversions} for an illustration of the zero set in the case $W_N=A_N$.

\begin{prop}\label{prop:DeltaInv}
	The assumptions of Theorem~\ref{thm:sc_main_theorem} are satisfied with
	\[
		\Delta_N=\begin{cases}
			\frac{\pi^2}{12}\sqrt{\frac73}\frac{(2N^3+9N^2+7N)^{1/2}}{N+1}&:\text{type }A_N\\
			\frac{\pi^2}{12}\sqrt{\frac76}\frac{(4N^3+6N^2-N)^{1/2}}{N}&:\text{type }B_N\\
			\frac{\pi^2}{12}\sqrt{\frac76}\frac{(4N^3-3N^2-N)^{1/2}}{N-1}&:\text{type }D_N,
		\end{cases}
		\label{eq:example_coxeter_delta_inequality}
	\]
	so that in all these cases, $\Delta_N= \mathcal{O}(\sqrt{N})$. In particular, $\sqrt{\Delta_{N}}X_N^\ast$ converges in the mod-Gaussian sense with parameters $\Delta_N$ and limiting function $\Psi(z) = \exp(-7\pi^4z^4/2400)$ independently of the type $W\in\{A,B,D\}$.
\end{prop}

To see this, note that the cumulants of $X_N$ admit the representation \eqref{eq:sc_cumulant_representation}. In particular, for each of the three possible choices for $W$ we consider, elementary calculations lead to expressions for $\sigma_N^2$, $M_N$ and $\kappa_{4,N}^\ast$ which are summarized in the following table. In all these cases, one moreover easily verifies that $\kappa_{4,N}^\ast\Delta_N^2\to-7\pi^4/100$ as $N\to\infty$, which yields mod-Gaussian convergence by Corollary~\ref{cor:modGauss}~\textit{(ii)}.

\begin{table}[h!]
	\renewcommand{\arraystretch}{1.5}
	\centering
	\begin{tabular}{c|l|l|l}
		Type & $\sigma_N^2$ & $M_N$ & $\kappa_{4,N}^*$\\
		\hline
		$A_N$&$\frac{1}{72} \left(2 N^3+9 N^2+7 N\right)$&$N+1$ & $-\frac{36}{25} \frac{6 N^5+45 N^4+130 N^3+180 N^2+89 N}{4 N^6+36 N^5+109 N^4+126 N^3+49 N^2}$\\
		$B_N$&$\frac{1}{36} \left(4 N^3+6 N^2-N\right)$&$2N$ & $-\frac{18}{25} \frac{48 N^5+120 N^4+80 N^3-23 N}{16 N^6+48 N^5+28 N^4-12 N^3+N^2}$\\
		$D_N$&$\frac{1}{36} \left(4 N^3-3 N^2-N\right)$&$2N-2$ & $-\frac{18}{25} \frac{48 N^4-105 N^3+80 N^2-23}{16 N^5-24 N^4+N^3+6 N^2+N}$\\
	\end{tabular}
\end{table}
\renewcommand{\arraystretch}{1}

The central limit theorem which follows from Proposition~\ref{prop:DeltaInv} agrees with the results found in \cite[Theorem~2.8, Corollary~2.9]{meier2022}, where generalized inversions in finite Weyl groups are studied. Figure~\ref{fig:clt_plot_inversions} shows the number of inversions in a sample of $350,000$ random permutations of $N=50$ elements (remind that the Coxeter group of Type $A_N$ corresponds to the symmetric group on $\{1,\dots,N+1\}$). We note in this context that in Remark 6.9 in \cite{kahle2020} the authors already mention without proof and without giving details a possible mod-Gaussian convergence for the sequence of random variables $X_N$.

\begin{figure}[h!]
	\includegraphics[width=0.5\textwidth]{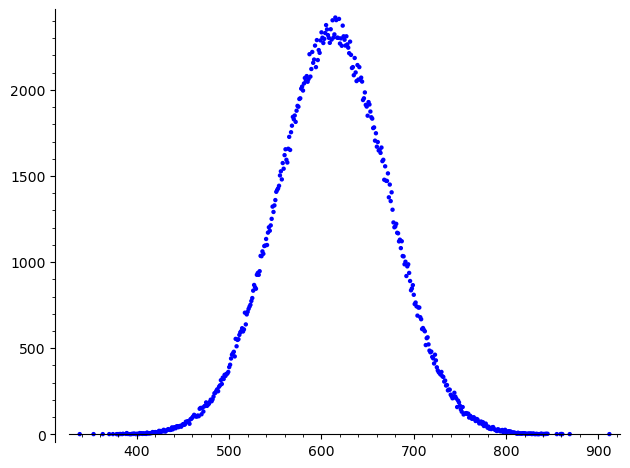}
	\caption{Simulation of the number of inversions of a permutation of $N=50$ elements, drawn from $350,000$ samples.}
	\label{fig:clt_plot_inversions}
\end{figure}

\subsubsection{Gaussian polynomials}
\label{ex:gaussian_polynomials}

Let $\gausspolcoeff$ denote the number of partitions of an integer $j$ into at most $\ell$ summands, each of which is less than or equal to $N$ (in particular, $p(N,\ell,N\ell)=1$ and $\gausspolcoeff=0$ if $j>N\ell$). According to \cite[Theorem~3.1]{andrews1998}, the corresponding generating polynomial $\gausspol$ is given by
\begin{equation}
	\gausspol=\sum_{0\le j \le N \ell}\gausspolcoeff z^j=\prod_{1\le j\le N}\frac{1-z^{j+\ell}}{1-z^j}\qquad(z\in\R).
	\label{eq:example_gaussian_generating_pol}
\end{equation}
These polynomials were historically first studied by Gauss and hence are called Gaussian polynomials, see for example \cite{andrews_gauss_1991}.

The right hand side of \eqref{eq:example_gaussian_generating_pol} matches \eqref{eq:sc_gen_polynomial} with $b_j=\ell+j$ and $a_j=j$, and thus $G(N,\ell;z)$ is a polynomial of degree $n=\sum_{1\le j\le N}\ell=N\ell$ (note the additional dependence on the parameter $\ell$). In particular, the generating function of Gaussian polynomials does not factorize as required in \cite[Lemma 4]{hwang1998}, so that the latter result cannot be applied in this situation in order to deduce a Berry-Esseen bound for the coefficients of a Gaussian polynomial.

\begin{prop}
	The assumptions of Theorem~\ref{thm:sc_main_theorem} are satisfied with $\sigma_{N,\ell}^2=(\ell^2 N+\ell N+\ell N^2)/12$, $M_{N,\ell}=N+\ell$ and
	\[
	\Delta_{N,\ell} =\frac{\pi^2}{6}\sqrt{\frac{7}{2}}\;\frac{\sqrt{\ell N (\ell+N+1)}}{ (\ell+N)} = \mathcal{O}\big( N^{-1/2}+\ell^{-1/2} \big).
	\]
	In particular, $\sqrt{\Delta_{N,\ell}}X_{N,\ell}^\ast$ converges in the mod-Gaussian sense as $N,\ell \to \infty$ with parameters $\Delta_{N,\ell}$ and limiting function $\Psi(z) = \exp(-35K\pi^4z^4/10368)$, where $K \in \mathbb{R}$ is the limit of $1+ (\ell+N-\ell N)/(N+\ell)^2$ as $N,\ell \to \infty$.
\end{prop}

To see this, note that by \eqref{eq:sc_cumulant_representation} the cumulants are given by
\begin{align}	
	\label{eq:example_gaussian_cum_representation}
	\kappa_{m,N,\ell}&=\frac{ B_{m}}{m}\sum_{1\le j\le N}\big((j+\ell)^{m}-j^{m}\big)\\
	&=\frac{ B_{m}}{m}\left(-H_N^{(-m)}-\zeta (-m,\ell+N+1)+\zeta (-m,\ell+1)\right),\nonumber
\end{align}
where $H_N^{(-m)}=\sum_{0\le k\le N}k^m$ is the $N$-th harmonic number of order $-m$ and $\zeta(-m,a)$ denotes the Hurwitz zeta function $\zeta(-m,a)=-\mathcal{B}_{m+1}(a)/(m+1)$ for any $a\in\R$ and
with $\mathcal B_{m}(\,\cdot\,)$ denoting the $m$-th Bernoulli polynomial. Using that $\mathcal B_3(a)=a^3-\frac{3}{2}a^2+\frac{1}{2}a$ (cf.\ e.\,g.\ \cite[Chapter~23]{abramowitz1964}), one readily calculates the variance.

Similarly, using \eqref{eq:example_gaussian_cum_representation} as well as $\mathcal B_5(a)=a^5-{5\over 2}a^4+{5\over 3}a^3-{1\over 6}a$, we find
\begin{align*}
	&\kappa_{4,N,\ell}=-\frac{1}{120} \left(-\frac{N^5}{5}-\frac{N^4}{2}-\frac{N^3}{3}+\frac{N}{30}+\frac{\mathcal{B}_{5}(\ell+N+1)}{5}-\frac{\mathcal{B}_{5}(\ell+1)}{5}\right),
\end{align*}
which leads to 
\begin{equation*}
	\kappa_{4,N,\ell}^\ast=-\frac{6}{5} \left(\frac{1}{N}+\frac{1}{\ell}-\frac{1}{\ell+N+1}\right),
\end{equation*}
after straightforward simplifications. As a consequence,
\begin{equation*}
	\kappa_{4,N,\ell}^\ast \Delta_{N,\ell}^2
	= - \frac{35\pi^4}{432} \Big(1 + \frac{\ell+N-\ell N}{(\ell+N)^2}\Big) \longrightarrow - \frac{35\pi^4}{432} K
\end{equation*}
as both $N\to\infty$ and $\ell\to\infty$. For instance, if $\ell = N$, we have $K= 3/4$.

Note that a central limit theorem for the random variables $X_{N,\ell}$ has previously been shown in \cite[Section~4]{mann1947} and \cite[Section~3]{takacs1986}, while the other results (especially the Berry-Esseen bound and mod-Gaussian convergence) seem new.

\subsubsection{Catalan numbers}
\label{ex:catalan_numbers}

There are several ways of generalizing the usual Catalan numbers
$C_N=\frac{1}{N+1}\binom{2N}{N}$ to $q$-Catalan numbers. One of them reads
\begin{equation*}
	C_N(q)=\frac{1}{[N+1]_q}\left[{2N}\atop{N}\right]_q,
\end{equation*}
where $[N]_q:=1+q+\dots+q^{N-1}$ and 
\[
\left[{{N}\atop{k}}\right]_q:=\frac{[N]_q!}{[k]_q![N-k]_q!},
\]
with $[N]_q! := [N]_q[N-1]_q \cdots [1]_q$. For a broader view on $q$-Catalan numbers we refer the reader to \cite{haiman1998} and to \cite{carlitz_riordan_1965} for further background information. In the sequel, we assume $N\ge2$. It is easy to check that
\begin{equation*}
	C_N(q)=\prod_{2\le j\le N}\frac{1-q^{N+j}}{1-q^{j}}.
\end{equation*}
which shows that $q$-Catalan numbers are closely related to the Gaussian polynomials \eqref{eq:example_gaussian_generating_pol}. Indeed, $C_N(q) = G(N,N;q)/[N+1]_q$. In particular, $C_N(q)$ is a polynomial as in \eqref{eq:sc_gen_polynomial} with $b_1=1$, $a_1=1$, $b_j=N+j$ and $a_j=j$ for $2\le j\le N$.

\begin{prop}
	The assumptions of Theorem \ref{thm:sc_main_theorem} are satisfied with $\sigma_N^2 = (N^3-N)/6$, $M_N = 2N$ and
	\[
	\Delta_N=\frac{\pi^2\sqrt{7}}{12} \sqrt{\frac{N^2-1}{N}} = \mathcal{O}(\sqrt{N}).
	\]
	In particular, $\sqrt{\Delta_{N}}X_N^\ast$ converges in the mod-Gaussian sense with parameters $\Delta_{N}$ and limiting function $\Psi(z) = \exp(-7\pi^4z^4/1920)$.
\end{prop}

To see this, note that the cumulants are given by
\begin{equation}
	\kappa_{m,N}=\frac{B_m}{m}\sum_{2\le j\le N}\big( (N+j)^m-j^m \big),
	\label{eq:example_q-catalan_kappa}
\end{equation}
so that the expressions for $\sigma_N^2$ and $\Delta_N$ directly follow. Moreover, mod-Gaussian convergence follows from noting that by \eqref{eq:example_q-catalan_kappa},
\begin{equation*}
	\kappa_{4,N}=-\frac{1}{60}\big(3N^5+3N^4-N^3-3N^2-2N\big),\qquad
	\kappa_{4,N}^\ast=-\frac{3}{5}\Big( \frac{3N^2+3N+2}{N^3-N} \Big),
\end{equation*}
and hence we obtain
\begin{equation*}
	\kappa_{4,N}^\ast\Delta_N^2=-\frac{7\pi^4}{270}\frac{3N^2+3N+2}{N^2} \longrightarrow-\frac{7\pi^4}{80}\qquad \text{as } N \to \infty.
\end{equation*}

Central limit theorems for $(X_N^\ast)_{N\ge1}$ have previously been shown in \cite[Corollary 3.3]{chen2008}, see also \cite[Theorem 3.1]{billey_konvalinka_swanson2020}. Moreover, our results recover the moderate deviation principle established in \cite[Theorem 2.1]{wu2014}, while they answer the question for a Berry-Esseen bound for $q$-Catalan numbers which was raised in \cite[Remark 4.2]{wu2014}.

\begin{remark}
	Generalizing the usual Catalan numbers $C_N$ in another way, one may define for integers $k\geq 1$ the $k$-Catalan numbers (or Fuss--Catalan numbers) by
	\begin{equation*}
		C_{N,k}=\frac{1}{(k-1)N+1}\binom{kN}{N},
	\end{equation*}
	see \cite{StanleyBook}. Given the $q$-Catalan numbers, it is natural to define
	\begin{equation*}
		C_{N,k}(q):=\frac{1}{[(k-1)N+1]_q}\left[{kN}\atop{N}\right]_q = \prod_{2\le j\le N}\frac{1-q^{(k-1)N+j}}{1-q^j},
	\end{equation*}
	using the same notation as before. The second representation is a polynomial of the form \eqref{eq:sc_gen_polynomial}. Hence, generalizing the results from above, one can show that
	\begin{equation*}
		\Delta_{N,k}=\sqrt{\frac{7}{6}}\;\pi^2\frac{\sigma_{N,k}}{M_{N,k}}=\pi^2\sqrt{\frac{7}{2}}\frac{ \sqrt{(k-1) (N-1) N (k N+2)}}{6 k N}.
		%	\label{eq:example_generalized_q-catalan_delta}
	\end{equation*}
	In particular, for any choice of $k > 1$ by Corollary~\ref{cor:clt_&_berry_esseen_bound}, the sequence $(X_N^\ast)_{N\ge1}$ satisfies a central limit theorem as $N \to \infty$ as previously shown in \cite[Section 3]{chen2008}. Noting that $\Delta_{N,k}$ can always be chosen of order $\mathcal{O}(N^{-1/2})$, all the other results from Section~\ref{ch:main_results} follow readily.
\end{remark}

\subsubsection{Descending plane partitions}
\label{ex:descending_plane_partitions}

A descending plane partition (DPP for short and first introduced by \cite{andrews1979}) is an array of non-negative integers $\gamma_{i,j}$ ($i\le j$), which are called the parts of the DPP, as in Figure~\ref{fig:descending_plane_partition}, such that the following conditions hold:
\begin{enumerate}
	\item[(D1)] The values of the parts are decreasing in each row from left to right and strictly decreasing in each column from top to bottom. In particular, $\gamma_{i,i}$ is the largest part of the $i$-th row and the $i$-th column.
	\item[(D2)] The entry $\gamma_{i,i}$ is strictly greater than the number of parts in the $i$-th row and less or equal to the number of parts in the $(i-1)$-th row.
\end{enumerate}
We refer to \cite[Section~1]{striker2011} or \cite[Section~1]{mills1983} for a detailed introduction. A descending plane partition is said to be of order $N$, if its largest part is at most $N$. For instance, there are two DPPs of order $N=2$ and seven DPPs of order $N=3$. Note that $\emptyset$ always counts as a DPP (the empty one). If $\gamma_{i,j} \le j-i$, $\gamma_{i,j}$ is called a special part of the DPP. See Figure~\ref{fig:descending_plane_partition_order_25} for a DPP of order $N=12$ (and thus also of any order $N>12$), where the specials parts are marked in bold.

Writing $S := \sum_{i,j} \gamma_{i,j}$ for the sum of the parts, a DPP of order $N$ can be regarded as a partition of $S$ whose largest part is at most $N$, see especially \cite[Lemma 5]{striker2011} which translates (D1) and (D2) into conditions on the partition for DPPs with no special parts. Descending plane partitions moreover admit connections to other combinatorial structures. For instance, in \cite{ayyer2010} a bijection between DPPs of order $N$ with no special parts and certain permutations of $N$ elements is established. There is also a relation of DPPs to alternating sign matrices as studied in \cite{mills1983,mills1982}.

\begin{figure}[t!]
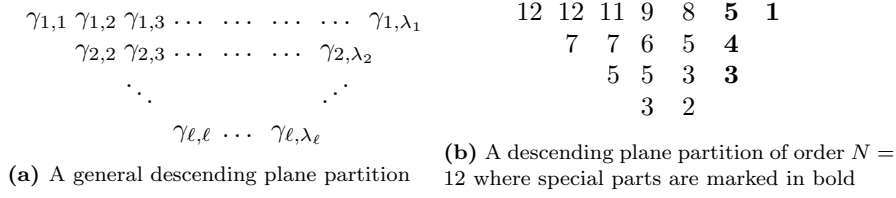

	\begin{minipage}{6cm}
		\centering
		\begin{tabular}{p{0.3cm}p{0.3cm}p{0.3cm}p{0.3cm}p{0.3cm}p{0.3cm}p{0.3cm}p{0.3cm}}
			$\gamma_{1,1}$ & $\gamma_{1,2}$ & $\gamma_{1,3}$ & \dots & \dots & \dots & \dots & $\gamma_{1,\lambda_1}$ \\
			& $\gamma_{2,2}$ & $\gamma_{2,3}$ & \dots & \dots & \dots & $\gamma_{2,\lambda_2}$ &\\
			& & $\ddots$ & & & & \reflectbox{$\ddots$}& \\
			& & & $\gamma_{\ell,\ell}$ & $\dots$ & $\gamma_{\ell,\lambda_\ell}$
		\end{tabular}
		\subcaption{A general descending plane partition}
		\label{fig:descending_plane_partition}
	\end{minipage}
	\begin{minipage}{6cm} 
		\centering
		\begin{tabular}{p{0.2cm}p{0.2cm}p{0.2cm}p{0.2cm}p{0.2cm}p{0.2cm}p{0.2cm}p{0.2cm}}
			$12$&$12$&$11$&$9$&$8$&$\boldsymbol{5}$&$\boldsymbol{1}$&\\
			&$\;7$&$\;7$&$6$&$5$&$\boldsymbol{4}$&&\\
			&&$\;5$&$5$&$3$&$\boldsymbol{3}$&&\\
			&&&$3$&$2$&&&\\
		\end{tabular}
		\subcaption{A descending plane partition of order $N=12$ where special parts are marked in bold}
		\label{fig:descending_plane_partition_order_25}
	\end{minipage}
	\caption{Descending plane partitions}
\end{figure}

According to \cite[Corollary~6]{striker2011}, the generating polynomial for DPPs of order $N$ with no special parts is given by
\begin{equation}
	\prod_{1\le j\le N}[j]_{q^j}=\prod_{1\le j\le N}\frac{1-q^{j^2}}{1-q^j},
	\label{eq:descending_plane_partitions_generating_function}
\end{equation}
where we recall the notation $[j]_q=1+q+q^2+\dots+q^{j-1}$ from Section \ref{ex:catalan_numbers}. This generating polynomial also appears in the context of the inversion index of a permutation of $N$ elements as shown in \cite{striker2011} using permutations matrices (see also [\href{http://www.findstat.org/St000616}{St000616}] in \cite{findstat}). In particular, \eqref{eq:descending_plane_partitions_generating_function} matches \eqref{eq:sc_gen_polynomial} with $b_j=j^2$ and $a_j=j$.

\begin{prop}
	The assumptions of Theorem~\ref{thm:sc_main_theorem} are satisfied with $\sigma_N^2 = (2 N^5+5 N^3-5 N^2-2N)/120$, $M_N=N^2$ and
	\begin{equation*}
		\Delta_N=\frac{\pi^2}{12}\sqrt{\frac{7}{5}}\;\frac{\left(2 N^5+5 N^3-5 N^2-2N\right)^{1/2}}{N^2} = \mathcal{O}(\sqrt{N}).
	\end{equation*}
	In particular, $\sqrt{\Delta_{N}}X_N^\ast$ converges in the mod-Gaussian sense with parameters $\Delta_{N}$ and limiting function $\Psi(z) = \exp(-7\pi^4z^4/2592)$.
\end{prop}

Indeed, the cumulants are given by
\begin{equation*}%\label{cumrepDPP}
	\kappa_{m,N}=\frac{B_m}{m}\sum_{1\le j\le N}\big( j^{2m}-j^m \big),
\end{equation*}
so that one readily calculates the variance and $\Delta_N$. Furthermore, we have
\begin{align*}
	\kappa_{4,N}&=-\frac{1}{2160}\left(2 N^9+9 N^8+12 N^7-12 N^5-9 N^4-2 N^3\right),\\
	\kappa_{4,N}^\ast&=-\frac{20 (N^2+N)}{6 N^3+9 N^2-9 N-6},
\end{align*}
and hence,
\begin{equation*}
	\kappa_{4,N}^\ast\Delta_N^2=-\frac{7\pi^4}{36}\frac{(N^2+N)(2 N^5+5 N^3-5 N^2-2N)}{N^4(6 N^3+9 N^2-9 N-6)}  \longrightarrow -\frac{7\pi^4}{108},
\end{equation*}
as $N\to\infty$, which yields mod-Gaussian convergence.

All these results seem new. In particular, Corollary \ref{cor:clt_&_berry_esseen_bound} provides a central limit theorem for the number of DPPs with largest part at most $N$ with speed of convergence of order $\mathcal{O}(N^{-1/2})$. See Figure~\ref{fig:clt_plot_inversion_index} for a plot of $350,000$ samples with $N=50$. The Gaussian shape is clearly visible.

\begin{figure}[t!]
	\includegraphics[width=0.5\textwidth]{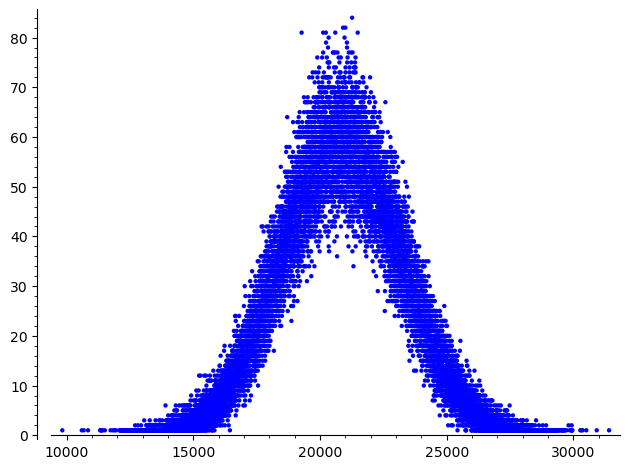}
	\caption{Simulation of the number of entries in a descending plane partitions of order $N=50$, drawn from $350,000$ samples.}
	\label{fig:clt_plot_inversion_index}
\end{figure}

% ================================================
\section{Proofs}
\label{ch:proofs}
% ================================================

\subsection{Proof of Theorem~\ref{thm:cumulant_bound}}
\label{ch:proof_general_cumulant_bound}
	
	The strategy is to apply the cumulant bound provided by \cite[Lemma~8.1]{michelen2019Preprint} and its proof. To this end, we set $u_N(z):=u_{X_N}(z)=\log\verts{f_N(z)}$ for each $X_N$ as well as
	\begin{equation*}
		S_R(\delta_N):=\big\{ z\in\C:\verts{z}\in[R^{-1},R],\arg(z)\in[-\delta_N,\delta_N]\big\}
	\end{equation*} 
	for $R > 1$.
	By assumption, $f_N$ is an analytic function on the complex plane which is zero-free in $S(\delta_N)$, and hence, $u_N$ is a harmonic function in a neighbourhood of $S_R(\delta_N/2)$ for any $R > 1$. Furthermore, $u_N$ is invariant under complex conjugation since
	\begin{equation*}
		u_N(\bar z)=\log\verts{f_N(\bar z)}=\log\verts{\overline{f_N(z)}}=\log\verts{f_N(z)}=u_N(z),
	\end{equation*}
	using that $f_N$ is a polynomial with real coefficients.
	
	Moreover, we have $u_N(|z|) \ge u_N(z)$ for any $z \in S(\delta_N)$,  since
	\begin{equation}\label{weaklypos}
		\log\verts{f_{X_N}(z)}=\log\verts{\E z^{X_N}}\le\log\E\verts{z}^{X_N}=\log\verts{f_{X_N}(\verts{z})}
	\end{equation}
    (this property is called \emph{weak positivity} in \cite{michelen2019Preprint}). In other words, on any circle centred at the origin, $u_N$ takes its maximum on the positive real axis. To apply \cite[Lemma 8.1]{michelen2019Preprint}, one needs a sharpening of this property and show that in a suitable neighbourhood of $1 \in \C$, $u_N$ decreases counter-clockwise (and due to the invariance under complex conjugation also clockwise) starting at the positive real axis on any circle centred at the origin. That is,
    \begin{equation}\label{0-decr}
    u(\rho e^{i \theta_1}) \ge u(\rho e^{i \theta_2})
    \end{equation}
    for any $\rho > 0$ and $0 \le \theta_1 \le \theta_2 \le \pi$ in some ball around $1 \in \C$ and such that $\rho e^{i\theta_1}$, $\rho e^{i \theta_2} \in S(\delta_N)$ (this property is called \emph{$0$-decreasing} in \cite{michelen2019Preprint}).
    
    We will show \eqref{0-decr} for $\rho e^{i \theta_1}, \rho e^{i \theta_2} \in S_2(\delta_N/2)$. By \cite[Lemma 4.1]{michelen2019Preprint}, since $u_N$ satisfies \eqref{weaklypos}, is invariant under complex conjugation and is analytic in a neighbourhood of $S_R(\delta_N)$ for any $R > 1$, a sufficient condition for \eqref{0-decr} to hold in $S_2(\delta_N/2)$ is given by
    \begin{equation}\label{cond}
    	\Big(\frac{2}{R}\Big)^{1/\delta_N}\max_{\substack{z\in S_R(0,\delta_N)\\\verts{z}\in\{R^{-1},R\}}}\verts{u_N(z)}\longrightarrow 0
    \end{equation}
	as $R \to \infty$. Noting that $f_N$ is a polynomial, we have $u_N(z) = \mathcal{O}(\log|z|)$ and $u_N(1/z) = \mathcal{O}(\log|z|)$ as $z \to \infty$. In particular, this implies that \eqref{cond} holds true since its left hand side is of order $\mathcal{O}(R^{-1/\delta_N} \log(R))$ as $R \to \infty$.
	
	Therefore, we may apply \cite[Lemma 8.1]{michelen2019Preprint} with the choice $\varepsilon=\delta_N/4$, noting that the condition
	$\sum_{j\ge2}\verts{\kappa_j}(j!)^{-1}(\varepsilon/32)^j>0$ therein is trivially satisfied. This yields
	\begin{equation*}
		\verts{\kappa_m^\ast}\le {m! \over (c\delta_N\sigma_N)^{m-2}}
	\end{equation*} 
	for all $m\ge3$ with $c= {1\over 4}\cdot 2^{-3246} = 2^{-3248}$. \qed

\subsection{Proof of Theorem~\ref{thm:sc_main_theorem}}
\label{ch:proof_of_sc_theorem}

\begin{lem}
	\label{lem:inequality}
	For any two real numbers $b\ge a\ge0$ and any integer $m\ge2$, it holds that
	\begin{equation}
		b^{2m}-a^{2m}\le(b^2-a^2)2^{m-1}b^{2m-2}.
		\label{eq:lem_inequality}
	\end{equation}
	\begin{proof}
		First note, that if $a=b$, \eqref{eq:lem_inequality} is trivial. In the case $b>a$, the statement can be proven by induction over $m$. The case $m=2$ is clear and we are left to show that if \eqref{eq:lem_inequality} holds for some $m\in\N$, then also
		\begin{equation*}
			b^{2m+2}-a^{2m+2}\le(b^2-a^2)2^mb^{2m}.
		\end{equation*} 
		Since $b\ge a$, we obtain
		\begin{equation*}
			\frac{b^{2m+2}-a^{2m+2}}{b^{2m}-a^{2m}}=b^2+\frac{a^{2m}b^2-a^{2m+2}}{b^{2m}-a^{2m}}\le b^2+\frac{a^2b^{2m}-a^{2m+2}}{b^{2m}-a^{2m}}\le2b^2,
		\end{equation*}
		which yields 
		\begin{align*}
			b^{2m+2}-a^{2m+2}&\le(b^{2m}-a^{2m})2b^2\le(b^2-a^2)2^mb^{2m},
		\end{align*}
		where the last step follows by induction.
	\end{proof}
\end{lem}

\begin{proof}[Proof of Theorem~\ref{thm:sc_main_theorem}]
	For $m\ge2$, the cumulant representation \eqref{eq:sc_cumulant_representation}, the identity $\kappa_{2m,N}^\ast=\kappa_{2m,N}/\sigma_N^{2m}$ and our assumption that $b_j\ge a_j$ for all $j\in\{1,\ldots,N\}$ give
	\begin{equation*}
		\vert\kappa_{2m,N}^\ast\vert=\frac{\vert B_{2m}\vert}{2m\sigma_N^{2m}}\sum_{1\le j\le N}(b_j^{2m}-a_j^{2m}).
	\end{equation*}
	An application of Lemma~\ref{lem:inequality} then leads to
	\begin{align*}
		\vert\kappa_{2m,N}^\ast\vert&\le\frac{\vert B_{2m}\vert}{2m\sigma_N^{2m}}\sum_{1\le j\le N}2^{m-1}b_j^{2m-2}(b_j^2-a_j^2)\\
		&\le\frac{2^{m-2}}{m\sigma_N^{2m}}\vert B_{2m}\vert M_N^{2m-2}\sum_{1\le j\le N}(b_j^2-a_j^2),
	\end{align*}
	where we recall that $M_N=\max_{1\le j\le N}b_j$ as stated in the theorem. Again by \eqref{eq:sc_cumulant_representation}, the variance can be written as
	\begin{equation*}
		\sigma_N^2=\kappa_{2,N}=\frac{B_2}{2}\sum_{1\le j\le N}(b_j^2-a_j^2).
	\end{equation*}
	Noting that $B_2=1/6$, this yields
	\begin{equation*}
		\vert\kappa_{2m,N}^\ast\vert\le \frac{12\cdot2^{m-2}}{m}\vert B_{2m}\vert\Big(\frac{M_N}{\sigma_N}\Big)^{2m-2}.
	\end{equation*}
	To bound the Bernoulli numbers $B_{2m}$ we use \cite[Equation 3.1.15]{abramowitz1964}, which states that
	\begin{equation*}
		\vert B_{2m}\vert\le\frac{2(2m)!}{(2\pi)^{2m}}\frac{1}{1-2^{1-2m}}\qquad(m\ge1).
		\end{equation*}
		As a consequence,
		\begin{align*}
			\vert\kappa_{2m,N}^\ast\vert&\le \frac{12\cdot2^{m-2}}{m}\frac{2(2m)!}{(2\pi)^{2m}}\frac{1}{1-2^{1-2m}}\Big(\frac{M_N}{\sigma_N}\Big)^{2m-2}\\
			&=(2m)!\frac{3\cdot2^{1-m}}{m\pi^{2m}(1-2^{1-2m})}\Big(\frac{M_N}{\sigma_N}\Big)^{2m-2}.
		\end{align*}
		The sequence $(c_m)_{m\ge2}$ defined by $c_m:=\frac{3\cdot2^{1-m}}{m\pi^{2m}(1-2^{1-2m})}$ is strictly decreasing, since
		\begin{equation*}
			c_m-c_{m+1}=3\pi ^{-2 m}2^{m+1} \left(\frac{1}{\left(4^m-2\right) m}-\frac{1}{\pi ^2 \left(2^{2m+1}-1\right) (m+1)}\right) >0.
		\end{equation*}
		To see this, note that for $m\ge2$,
		\begin{align*}
			\pi ^2(2^{2 m+1}-1)(m+1)&=m(4^m+1)\pi^2+(2^{2m+1}-1)\pi^2>m(4^m-2),
		\end{align*}
		and therefore, 
		\begin{align*}
			\frac{1}{m(4^m-2)}>\frac{1}{(m+1)(2^{2m+1}-1)\pi^2}.
		\end{align*}
		In particular, this implies that $\max_{m\ge2}c_m=c_2=6/(7\pi^4)$. Thus,
		\begin{align*}
			\vert\kappa_{2m,N}^\ast\vert&\le(2m)!\frac{6}{7\pi^4}\Big(\frac{M_N}{\sigma_N}\Big)^{2m-2}.
		\end{align*}
		So, choosing 
		\begin{equation*}
			\Delta_N=\max_{m\ge2}\Big\{\sqrt[2m-2]{\frac{7\pi^4}{6}}\Big\}\frac{\sigma_N}{M_N}=\sqrt{\frac{7\pi^4}{6}}\frac{\sigma_N}{M_N}
		\end{equation*}
		we arrive at
		\begin{equation*}
			\vert\kappa_{2m,N}^\ast\vert\le\frac{(2m)!}{\Delta_N^{2m-2}}.
		\end{equation*}
		This completes the proof.
\end{proof}

\appendix

\section{The method of cumulants}

In order to keep this paper self-contained, we briefly summarize the probabilistic consequences which can be drawn from the Statulevi\v{c}ius condition. This approach is known as the method of cumulants and we refer to the monograph \cite{saulis1991} as well as the recent survey \cite{doring2021} for a detailed account of this method. Recall that for a sequence of random variables $(X_N)_{N\geq 1}$, we denote the cumulant of order $m$ of $X_N$ by $\kappa_{m,N}:=\kappa_m(X_N)$ and the respective cumulant of the normalization $X_N^\ast:=(X_N-\E X_N)/\sqrt{\var(X_N)}$ by $\kappa_{m,N}^\ast:=\kappa_{m}(X_N^\ast)$. We present the result in a more general form, although we  had $\gamma=0$
in all our applications.

\begin{lem}
	\label{lem:appendix_1}
	Let $(X^\ast_N)_{N\geq 1}$ be a sequence of random variables which satisfies
	\begin{equation}\label{Statallgem}
		\verts{\kappa_{m,N}^\ast} \le \frac{(m!)^{1+\gamma}}{\Delta_N^{m-2}}
	\end{equation}
	for all $m\ge3$, all $N\ge1$, some $\gamma\ge0$ and $\Delta_N>0$. Assume that $\Delta_{N}\to\infty$ as $N\to\infty$.
	\begin{itemize}
		\item[(i)] For any $N\ge1$ it holds that
		\begin{equation*}
			\sup_{x\in\R}\verts{\pr(X^\ast_N\ge x)-\pr(Z\ge x)} \le \frac{C_\gamma}{\Delta_N^{1/(1+2\gamma)}},
		\end{equation*}
		where $C_\gamma\in(0,\infty)$ only depends on $\gamma$.
		
		\item[(ii)] For every sequence $(a_N)_{N\geq 1}$ with $a_N\to\infty$ and $a_N=o(\Delta_N^{1/(1+2\gamma)})$, the sequence of random variables $X_N^\ast/a_N$ satisfies a moderate deviation principle with speed $a_N^2$ and rate function $I(x)=x^2/2$.  
		
		\item[(iii)] Assume that $|\kappa_{m,N}^\ast| \le (m!/2)^{1+\gamma} H/\bar{\Delta}_N^{m-2}$ for all $m\ge 3$, $N \ge 1$, some $\gamma \ge 0$ and $H,\bar{\Delta}_N > 0$. Then, there exists a constant $C>0$ such that for every $N\ge1$ and all $x\ge0$,
		\begin{equation*}
			\pr(X_N^\ast\ge x)\le C\exp\Big( -\frac{1}{2}\frac{x^2}{H+x^{2-\alpha}/\bar{\Delta}_N^{\alpha}} \Big),
		\end{equation*}
		where $\alpha:=1/(1+\gamma)$. 
	\end{itemize}
\end{lem}

Statement \textit{(i)} of the above lemma corresponds to Theorem 2.4 in \cite{doring2021} and to Corollary 2.1 in \cite{saulis1991}. The latter gives the precise value
\[
	C_\gamma=2^{1/(4 \gamma+2)+2} \cdot 3^{1/(2 \gamma+1)+3}
	\]
for the constant. Statement \textit{(ii)} agrees with Theorem 3.1 in \cite{doring2021} and with Theorem 1.1 in \cite{doring2013}. Finally, \textit{(iii)} corresponds to Theorem 2.5 in \cite{doring2021} and to Corollary 2.3 in~\cite{saulis1991}. In our applications, we have set $H=2$, in which case (together with $\gamma=0$) the assumption on the cumulants agrees with \eqref{Statallgem}.

If $\gamma = 0$, \eqref{Statallgem} also gives rise to mod-Gaussian convergence. We state a corresponding result in two versions, one of them directly taken from the literature and the other one adapted to the situation we consider in the present paper.

\begin{lem}
	Let $(X_N)_{N\geq 1}$ be a sequence of random variables.
	\begin{itemize}
		\item[(i)] Assume that $|\kappa_{m,N}| \le (Cm)^m \alpha_N\beta_N^m$ for all $m \ge 2$ and for some sequences $(\alpha_N)_{N\geq 1} \to \infty$ and $(\beta_N)_{N\geq 1}$ satisfying $\alpha_N\to\infty$. Further assume that there exists an integer $v \ge 3$ such that $\kappa_{m,N} = 0$ for any $m = 3,\ldots,v-1$ and that
		\begin{align*}
			\kappa_{2,N} &= \sigma^2 \alpha_N \beta_N^2(1 + o(\alpha_N^{-(v-2)/v})),\\
			\kappa_{v,N} &= L\alpha_N \beta_N^v (1+o(1)).
		\end{align*}
		Then, the sequence of random variables $(X_N-\mathbb{E}[X_N])/(\alpha_N^{1/v}\beta_N)$ converges in the mod-Gaussian sense with parameters $t_N = \sigma^2 \alpha_N^{(v-2)/v}$ and limiting function $\psi(z) = e^{Lz^v/v!}$.
		\item[(ii)] Assume that \eqref{Statallgem} holds. Further assume that there exists an integer $v \ge 3$ such that $\kappa_{m,N}^\ast = 0$ for any $m=3, \ldots, v-1$ and that $\kappa_{v,N}^\ast \Delta_N^{v-2}\to L \in \mathbb{R}$ as $N \to \infty$. Then, the sequence of random variables $\Delta_N^{1-2/v} X_N^\ast$ converges in the mod-Gaussian sense with parameters $\Delta_N^{2(v-2)/v}$ and limiting function $e^{Lz^v/v!}$.
	\end{itemize}
\end{lem}

Part \textit{(i)} of Lemma \ref{lem:appendix_1} is a summary of the results of Section 5.1 in \cite{feray2016}, and \textit{(ii)} immediately follows from it as we may choose $C=1$, $\alpha_N=\Delta_{N}^2$, $\beta_N=\sigma_N/\Delta_{N}$ and $\sigma^2=1$.

\subsection*{Acknowledgment}
We are grateful to Martina Juhnke-Kubitzke, Kathrin Meier, Benedikt Redno\ss\  and Christian Stump for stimulating discussions about the content of this paper. We would moreover like to thank an anonymous reviewer of the first version of this paper who pointed us to the cumulant bound \eqref{cumulantbound} given in \cite{michelen2019Preprint}, largely generalizing our initial results for the class of polynomials discussed in Section \ref{ch:special case}. \\ The last author was supported by the German Research Foundation (DFG) via the priority program SPP 2265 and the collaborative research centre CRC/TRR 191.

\printbibliography

\end{document}

%-----------------------------------------------------------------------
% End of amsart-template.tex
%-----------------------------------------------------------------------